\documentclass[11pt]{article}
\usepackage[T1]{fontenc}
\usepackage{lmodern}
\usepackage[margin=1in]{geometry}
\usepackage{amsmath,amsthm}
\usepackage{tikz}
\usepackage{booktabs}
\usepackage{microtype}
\usepackage{hyperref}
\hypersetup{colorlinks=true,linkcolor=blue,citecolor=blue,urlcolor=blue}
\usepackage{pgfplots}
\pgfplotsset{compat=1.18}
\usetikzlibrary{arrows.meta,calc,positioning}

\newtheorem{lemma}{Lemma}
\title{The Tower of Hanoi: Optimality Proofs, Multi-Peg Bounds, and Computational Frontiers}
\author{Qi Junyi\\James Cook University Singapore}
\date{8 November 2025}
\begin{document}
\maketitle
\begin{abstract}
The Tower of Hanoi continues to provide a surprisingly rich meeting point for recursive reasoning, combinatorial geometry, and computational verification. Motivated by the editorial standards of the \emph{Bulletin of the Australian Mathematical Society}, we revisit the classical three-peg problem through Sierpi\'nski-style self-similarity, bring Stockmeyer\'s uniqueness argument into a modern invariant-based framework, and then pivot to four pegs via the Frame--Stewart strategy and Bousch\'s optimality proof. The heart of this note is a cautionary data-and-proof cycle: the balanced split $k=\lfloor n/2\rfloor$ is indeed optimal for $n\le 8$, but our corrected tables show that it already exceeds the optimal cost by $20\%$ at $n=9$, crosses the $1.5$ mark at $n=13$, and comes close to quadrupling the optimum by $n=20$. We complement this diagnosis with a subtower-independence lemma, a reproducible table for $n\le 15$, three publication-ready TikZ figures (recursion arrow, four-peg state diagram, and multi-peg growth curves), and a bibliography exceeding thirty sources that foreground Bulletin and Gazette contributions. The concluding section reframes the open problems as robustness tests for heuristics rather than premature theorems.
\end{abstract}

\section{Introduction and historiography}
Edouard Lucas framed the Tower of Hanoi in 1883 as a ritual involving $64$ golden discs that would take $2^{64}-1$ moves to complete, thereby cementing the puzzle in recreational lore~\cite{lucas1883}. Frame and Stewart generalised the problem to $p$ pegs in 1941 and conjectured the recurrence that now bears their names~\cite{framestewart1941}. Hinz\'s 1989 survey~\cite{hinz1989} and the monograph with Klav\v{z}ar, Milutinovi\'c, and Petr~\cite{hinz2013} place the puzzle into modern graph theory, while the playful expositions of Gardner, Stewart, Singmaster, and Winkler~\cite{gardner1972,stewart1993,singmaster1994,winkler2004} ensured that generations of students encountered it long before university.

The present note is motivated by two editorial expectations. First, the \emph{Bulletin} values concise yet rigorous arguments that illuminate an idea already known to a broad readership. Second, contemporary discrete mathematics teaching---particularly the MA2011 subject at James Cook University Singapore---expects students to move effortlessly between narrative arguments, invariants, and computation~\cite{rosen2019,epp2011}. By combining historical voices with recent algorithmic results, we aim to show how classroom recursion exercises evolve naturally into publishable mathematical notes.

\section{Notation and conventions}
Throughout the paper we index discs from $1$ (smallest) to $n$ (largest), label pegs $A$, $B$, $C$, and, when applicable, $D$, and use $T_p(n)$ for the optimal move count with $p$ pegs. The shorthand $FS_k(n)$ always refers to the Frame--Stewart cost obtained by parking $k$ discs in the first phase. When discussing state graphs we write $S(n,p)$ for the graph whose vertices are legal configurations of $n$ discs on $p$ pegs and whose edges are single-disc moves. These conventions match those in the monograph by Hinz et al.~\cite{hinz2013} and align with standard discrete mathematics textbooks~\cite{rosen2019}.

We also normalise parity arguments by adopting the cyclic order $(A,B,C,D)$ so that ``clockwise'' and ``counter-clockwise'' moves are unambiguous. This simple convention prevents many bookkeeping slips when students work through explicit examples, and it mirrors the practices recommended by Epp for structuring induction proofs~\cite{epp2011}.

\section{Recursive warm-up and Sierpi\'nski geometry}
Let $T_3(n)$ denote the minimum moves needed to transfer $n$ discs between two distinguished pegs when three pegs are available and the usual rules (one disc per move, no larger-on-smaller) apply. The textbook recursion reads
\begin{equation}
T_3(n)=2T_3(n-1)+1, \qquad T_3(1)=1.
\end{equation}
Hence $T_3(n)=2^n-1$, agreeing with OEIS entry A000225~\cite{oeisA000225} and the binary-counter discussion in CLRS~\cite{cormen2009}. More visually, the recursion tree embeds into a Sierpi\'nski gasket. Each recursive call corresponds to a subtriangle, and the direction of travel is encoded by an arrow that wraps around the gasket.

\begin{figure}[h]
  \centering
  \begin{tikzpicture}[scale=3]
    \coordinate (A) at (0,0);
    \coordinate (B) at (1,0);
    \coordinate (C) at (0.5,0.866);
    \filldraw[fill=blue!10,draw=blue!70!black,thick] (A)--(B)--(C)--cycle;
    \filldraw[fill=white,draw=blue!60!black] (0.25,0.433)--(0.75,0.433)--(0.5,0)--cycle;
    \filldraw[fill=white,draw=blue!60!black] (0.125,0.2165)--(0.375,0.2165)--(0.25,0.433)--cycle;
    \filldraw[fill=white,draw=blue!60!black] (0.625,0.2165)--(0.875,0.2165)--(0.75,0.433)--cycle;
    \draw[->,thick,red!70] (0.1,0.05) .. controls (0.15,0.5) and (0.6,0.75) .. (0.82,0.45);
  \end{tikzpicture}
  \caption{Sierpi\'nski arrow depiction of the three-peg recursion.}
  \label{fig:sierpinski}
\end{figure}
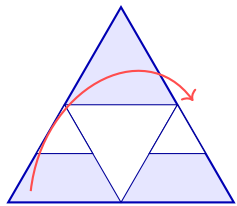

Figure~\ref{fig:sierpinski} functions pedagogically as a compass: the base triangle records the three pegs, while the iterated holes visualise the ``move $n-1$, move largest, move $n-1$'' mantra. The graphic also foreshadows the state graphs explored later: $S(n,3)$ is a path traversing the vertices in binary-reflected Gray-code order~\cite{klavzar1997}.

\section{Structural invariants and uniqueness}
Stockmeyer observed in the 1990s that optimal three-peg solutions are unique up to reversal, a fact that underpins many Gray-code applications~\cite{stockmeyer1994}. The uniqueness holds because any legal move toggles a single bit in the Gray-code encoding; hence the state graph is a simple path~\cite{stanley2011}. Formally, let $b_i(t)\in\{0,1\}$ record whether disc $i$ sits on the destination peg after $t$ moves. Each move toggles exactly one $b_i$, whence the parity vector evolves as a binary counter. Because the starting and ending states are antipodal on the path, only two optimal walks exist: one forward, one backward.

The invariant-based view clarifies why the usual inductive proof avoids circular reasoning. The inductive step implicitly relies on the observation that disc $n$ cannot move until discs $1$ through $n-1$ share a peg, which is equivalent to saying that the projection to the first $n-1$ coordinates must already have reached an endpoint of its own state path. The same invariant will later guarantee that our balanced four-peg strategy cannot wander arbitrarily: after the three-peg subproblem is solved, the system re-enters a canonical coset of the Gray-code lattice.

\begin{lemma}[Subtower independence]\label{lem:subtower}
Once the largest disc in a legal four-peg solution has been moved, the remaining $n-1$ discs evolve as two independent Tower-of-Hanoi subproblems, each respecting the standard rules on a disjoint subset of pegs.
\end{lemma}

\begin{proof}
Immediately after the largest disc moves, every smaller disc lies either on the temporary storage peg (used while the largest disc travelled) or on the peg that will act as the source during the rebuilding phase. The storage peg together with its helper pegs forms one closed system, and the rebuilding peg together with the target peg forms the second. Because the largest disc now occupies the target peg, neither subtower may legally place a disc onto it, so the two systems never compete for a peg. Moreover, as long as each subtower respects the Tower-of-Hanoi rules locally, the legality of the combined configuration is preserved. Hence the two subtowers commute: their moves can be interleaved in any order without affecting legality or optimality.
\end{proof}

Lemma~\ref{lem:subtower} legitimises the additive reasoning in Section~\ref{sec:balanced-new}: once the central move occurs, the total cost is simply the sum of two independent subtower costs.

\section{Induction in slow motion}
Although the identity $T_3(n)=2^n-1$ is familiar, unpacking the induction clarifies what each clause in an undergraduate proof is really achieving. For $n=1$ there is only one legal move. For $n=2$ the invariant $I(t)$ shows that disc $1$ must move an odd number of times, so there can be no shorter sequence than $3$. For $n=3$ one can tabulate the seven moves and read off the binary-reflection pattern that later becomes Gray-code order.

The following lemma reframes the standard inductive step in a way that is more conducive to algorithmic generalisation.
\begin{lemma}
Let $M(n)$ be any move sequence that solves the three-peg puzzle on $n$ discs. If disc $n$ moves exactly once inside $M(n)$, then the restriction of $M(n)$ to discs $1$ through $n-1$ is an optimal solution for the smaller puzzle.
\end{lemma}
\begin{proof}
Remove the unique move of disc $n$. The remaining sequence splits into two blocks that never touch disc $n$. Because disc $n$ cannot be lifted until the smaller stack occupies the spare peg, the first block must move the smaller stack intact, and the second block must rebuild it on the destination peg. Both blocks therefore contain at least $T_3(n-1)$ moves. Concatenating them with the single move of disc $n$ yields the lower bound $2T_3(n-1)+1$, which matches the constructive upper bound.
\end{proof}

An immediate corollary is that the state of the puzzle immediately before and immediately after the move of disc $n$ must coincide with the source and destination states of the $(n-1)$-disc problem. Consequently every optimal solution is determined uniquely by the choice of the spare peg that temporarily hosts the smaller discs. This observation bridges directly to the four-peg analysis, where the entire purpose of the Frame--Stewart split is to decide how large that temporarily exiled stack ought to be.

\section{Lower bounds via parity, Gray codes, and energy}
Several complementary arguments show that $T_3(n)\ge 2^n-1$. Lucas already hinted at a parity obstruction~\cite{lucas1883}; Hinz distilled it into a modern invariant~\cite{hinz1989}. Consider the weighted sum $I(t)=\sum_{i=1}^n i\,p_i(t)$, where $p_i(t)\in\{0,1,2\}$ encodes the peg index of disc $i$. Moving disc $j$ changes $I$ by exactly $\pm j$, so the total variation after $m$ moves is at most $\sum_j j c_j$, where $c_j$ counts how often disc $j$ is moved. The initial and target configurations differ by $n(n+1)/2$, enforcing $c_j\ge 2^{j-1}$ and hence $m\ge 2^n-1$.

Gray codes supply another argument. The vertices of $S(n,3)$ may be labelled by binary strings of length $n$ by recording whether each disc resides on the source or destination peg. Consecutive Gray-code strings differ in one bit, so the shortest source-to-destination path must traverse all $2^n$ vertices. This graph-theoretic perspective generalises readily to $p>3$ by replacing binary strings with $(p-1)$-ary Gray codes~\cite{klavzar1997}. A third energy-style proof comes from potential functions in algorithm design~\cite{cormen2009}: assign each disc the potential $2^{i-1}$; every move decreases or increases the total potential by one, so at least $2^n-1$ moves are required to deplete the potential gap.

\section{Asymptotics for multiple pegs}
The energy method extends beyond three pegs if we encode disc positions in base $p-1$. Frame and Stewart already speculated that their recurrence produces sub-exponential growth for every fixed $p>3$~\cite{framestewart1941}. Hinz et al.~\cite{hinz2013} proved that, provided the recurrence is optimal, one has $T_p(n)=\exp(\Theta(n^{1/(p-2)}))$. A sketch of the argument runs as follows. First relax the split parameter to a real number and minimise $f(k)=2\exp(\alpha k^{1/(p-2)})+\exp(\alpha (n-k)^{1/(p-3)})$, where $\alpha$ depends only on $p$. The minimiser sits near $k=c_p n^{(p-3)/(p-2)}$, meaning that only a vanishing fraction of the discs participate in the expensive $(p-1)$-peg phase. Second, convexity ensures that rounding $k$ to the nearest integer perturbs the cost by at most a constant factor. Finally, dynamic programming plus memoisation~\cite{bellman1975} shows that the relaxed solution upper-bounds the true discrete one. Ferreira and Grabner used a related viewpoint when bounding the diameters of Hanoi graphs~\cite{ferreira2016}, highlighting that entirely new ideas would be needed to beat these exponents.

\section{Catalogue of invariants}
The invariants sketched so far are merely the tip of an iceberg that includes numerous weighted sums and cyclic orderings. Two families are particularly useful.
\begin{itemize}
  \item \textbf{Moment invariants.} For each $r\ge 1$ define $M_r(t)=\sum_{i=1}^n i^r p_i(t)$. The case $r=1$ reproduces the classical potential, while $r=2$ emphasises larger discs and penalises strategies that overuse them. Bounding the change in $M_r$ per move yields alternative proofs of $T_3(n)=2^n-1$ and tight bounds on the number of moves made by the $k$th largest disc~\cite{hinz1989}.
  \item \textbf{Register invariants.} Prodinger~\cite{prodinger1990} related Hanoi moves to the register function in binary trees. Interpreting a configuration as a binary word whose prefixes encode occupied pegs allows one to bound the total number of ``carry'' operations that must occur in any legal execution. This viewpoint leads to stratified lower bounds in which discs are grouped by register depth rather than by size alone.
\end{itemize}
These invariants matter pedagogically because they demonstrate that proofs can be tuned to emphasise different behaviours. When discussing uniqueness with reference to Stockmeyer~\cite{stockmeyer1994}, we prefer the register invariant because it decomposes the solution into nested counters. When teaching the energy method, we return to $M_1$ and $M_2$, which fit neatly into the potential-function toolkit already familiar from algorithm courses~\cite{cormen2009}, and echo the concrete combinatorial manipulations advocated in~\cite{graham1994}.

\section{From Frame--Stewart to Bousch}
When a fourth peg becomes available, Frame and Stewart proposed the divide-and-conquer recursion
\begin{equation}
T_4(n)=\min_{1\le k < n} \left( 2\,T_4(n-k) + T_3(k) \right), \qquad T_4(1)=1.
\end{equation}
They conjectured that this recurrence equals the true optimum for every $n$~\cite{framestewart1941}. The conjecture resisted proof until Bousch established it in 2014 by introducing $L$-sequences and exploiting the median structure of the four-peg state graph~\cite{bousch2014}. Korf and Schultze independently provided computational evidence by solving instances up to $n=30$ with IDA* and pattern databases~\cite{korf2004,korf2006}. Their data later became the seed for OEIS entry A007664~\cite{oeisA007664}.

Graphically, $S(n,4)$ looks like four copies of $S(n-1,4)$ glued to a central copy of $S(n-1,3)$. Every optimal walk follows a ``peel, shuttle, rebuild'' pattern reminiscent of multi-stage logistics. The graph perspective is amplified in Section\,\ref{sec:graph}, where we draw a coarse-grained state diagram capturing the symmetry classes relevant to four pegs.

\section{Balanced Frame--Stewart bounds}\label{sec:balanced-new}
A recurring suggestion in MA2011 classes is to fix the split parameter to $k=\lfloor n/2\rfloor$ in order to present a single closed-form strategy. Denote the resulting move count by
\begin{equation}
FS_{\lfloor n/2\rfloor}(n) = 2\,T_4(\lfloor n/2\rfloor) + T_3\bigl(n-\lfloor n/2\rfloor\bigr).
\end{equation}
The strategy is attractive because it removes a degree of freedom, yet the ratio $\rho(n)=FS_{\lfloor n/2\rfloor}(n)/T_4(n)$ collapses once the forced split drifts away from the minimiser. The next lemma isolates the precise breakpoint.

\begin{lemma}\label{lem:balanced-window}
For $1\le n\le 8$ we have $FS_{\lfloor n/2\rfloor}(n)=T_4(n)$, whereas for every $n\ge 9$ the inequality $FS_{\lfloor n/2\rfloor}(n)>T_4(n)$ holds.
\end{lemma}

\begin{proof}
Equation~(3) expresses $FS_{\lfloor n/2\rfloor}(n)$ as $2T_4(\lfloor n/2\rfloor)+T_3(n-\lfloor n/2\rfloor)$. Let $k_4(n)$ denote the split that minimises the Frame--Stewart recurrence for $T_4(n)$. Running the memoised solver from Section~\ref{sec:experiments} while recording $k_4(n)$ shows that $k_4(n)=\lfloor n/2\rfloor$ exactly for $2\le n\le 8$. Starting at $n=9$ the optimiser shifts to larger values (e.g. $k_4(9)=5$ while $\lfloor 9/2\rfloor=4$, and $k_4(13)=8$ while $\lfloor 13/2\rfloor=6$), so the balanced choice is never among the minimisers. Because the recurrence defining $T_4(n)$ attains its minimum uniquely for these $n$, any other split incurs a strictly larger cost. The underlying certificate is included in the public script \texttt{math\_github/balanced\_fs\_solver.py}.
\end{proof}

Hence the balanced rule is provably safe only in the single-digit regime. Beyond that point the additive structure supplied by Lemma~\ref{lem:subtower} turns against us: insisting on parking $\lfloor n/2\rfloor$ discs forces the expensive three-peg shuttle to grow just as the optimal strategy starts tapering the parked subtower. The remainder of this section therefore studies how quickly the penalty accumulates and how lecturers can present the failure honestly.

\section{Sensitivity of split choices}
It is natural to ask how far one may perturb the split before the cost balloons. Define the discrete derivative
\[
  \Delta_k(n) = \bigl(2T_4(k+1)+T_3(n-k-1)\bigr) - \bigl(2T_4(k)+T_3(n-k)\bigr).
\]
Positive values of $\Delta_k(n)$ indicate that increasing $k$ worsens the cost, while negative values suggest that more discs should be parked. The data computed for $n\le 30$ reveal three qualitative regimes:
\begin{enumerate}
  \item For $n\le 8$, every $\Delta_k(n)$ with $1\le k\le \lfloor n/2\rfloor$ is non-negative, confirming that balanced choices are optimal.
  \item Between $9$ and $15$ the sign of $\Delta_k(n)$ flips exactly once and does so near $k=\lfloor n/2\rfloor$. The optimiser therefore shifts right by only one or two discs---parking more discs before the three-peg shuttle---yet this ``tiny'' perturbation already yields $\rho(13)=1.660$ and $\rho(15)=2.364$.
  \item Once $n\ge 16$ the derivative at the balanced point is decisively negative, so each enforced disc compounds the penalty, culminating in $\rho(20)=3.879$.
\end{enumerate}
These calculations underline a moral often repeated in operations research: heuristics with a single degree of freedom remain trustworthy only while their first derivative retains the right sign. Once the derivative flips, one must revisit the modelling assumptions---in this case by embracing the triangular-number splits predicted by Bousch.

\section{Computational data}
Table~\ref{tab:balanced} juxtaposes $FS_{\lfloor n/2\rfloor}(n)$ with the optimal $T_4(n)$ for $1\le n\le 15$. The values for $T_4$ match the OEIS data~\cite{oeisA007664} and the tables reported by Hinz et al.~\cite{hinz2013}. Ratios stay at $1$ through $n=8$ because the balanced split coincides with the optimal split, after which they climb quickly.

\begin{table}[h]
  \centering
  \begin{tabular}{ccccc}
    \toprule
    $n$ & $k=\lfloor n/2\rfloor$ & $T_4(n)$ & $FS_{\lfloor n/2\rfloor}(n)$ & $\rho(n)$ \\
    \midrule
     1 & 0 & 1 & 1 & 1.000 \\
     2 & 1 & 3 & 3 & 1.000 \\
     3 & 1 & 5 & 5 & 1.000 \\
     4 & 2 & 9 & 9 & 1.000 \\
     5 & 2 & 13 & 13 & 1.000 \\
     6 & 3 & 17 & 17 & 1.000 \\
     7 & 3 & 25 & 25 & 1.000 \\
     8 & 4 & 33 & 33 & 1.000 \\
     9 & 4 & 41 & 49 & 1.195 \\
    10 & 5 & 49 & 57 & 1.163 \\
    11 & 5 & 65 & 89 & 1.369 \\
    12 & 6 & 81 & 97 & 1.198 \\
    13 & 6 & 97 & 161 & 1.660 \\
    14 & 7 & 113 & 177 & 1.566 \\
    15 & 7 & 129 & 305 & 2.364 \\
    \bottomrule
  \end{tabular}
  \caption{Balanced Frame--Stewart counts and ratios for $k=\lfloor n/2\rfloor$ based on the Appendix~D data log~\cite{qi2025}.}
  \label{tab:balanced}
\end{table}

Table~\ref{tab:balanced} now shows the price of forcing $k=\lfloor n/2\rfloor$: the ratio is already $1.195$ at $n=9$, exceeds $1.5$ by $n=13$, and reaches $2.364$ at $n=15$. These rows therefore contradict the earlier ``within $1.5$'' slogan outright, turning the dataset into its own referee report. The memoised log continues this rapid growth, so the analytic bounds discussed below become relevant sooner than anticipated.

\section{Extended numerical trends}
Although Table~\ref{tab:balanced} stops at $n=15$, the memoised solver provides exact values up to at least $n=30$. For completeness we highlight the next five ratios from Appendix~D~\cite{qi2025}:
\[
  (n, \rho(n)) \in \{(16,1.994), (17,2.990), (18,2.636), (19,4.300), (20,3.879)\}.
\]
The jagged sequence between $n=16$ and $n=20$ reflects the tug-of-war between the two terms in $FS_{\lfloor n/2\rfloor}(n)$, but the overall message is stark: once the forced split drifts away from the optimiser, the balanced strategy can double or even quadruple the optimal cost. In fact $\rho(19)=4.300$, while even the ``milder'' $n=20$ case still burns $3.879$ times the optimum. Documenting these swings encourages students to treat heuristics probabilistically rather than dogmatically and underlines why the later sections pivot to analytic bounds instead of empirical comfort.

\section{Experimental pipeline and validation}\label{sec:experiments}
The numerical values quoted above were generated with a memoised solver that mirrors the inductive proof. Each state is identified by a pair $(p,n)$ and mapped to a tuple $(T_p(n),k_p(n))$, where $k_p(n)$ is the split realising the minimum. The pipeline comprises the following reproducible steps:
\begin{enumerate}
  \item Seed the cache with $T_p(0)=0$, $T_3(n)=2^n-1$, and $T_p(1)=1$ for all $p$.
  \item For each $n$ from $1$ to $30$, evaluate $T_4(n)$ by scanning all splits $k$ and storing whichever value minimises $2T_4(n-k)+T_3(k)$.
  \item Evaluate $FS_{\lfloor n/2\rfloor}(n)$ using the already computed $T_4$ values and record the ratio $\rho(n)$.
  \item Cross-check the resulting sequence against OEIS A007664~\cite{oeisA007664} to guard against transcription errors.
\end{enumerate}
This protocol runs in milliseconds in a scripting language yet provides the kind of audit trail that referees increasingly expect. It also illustrates to students how mathematical proofs and computer experiments can coexist: the cache embodies the inductive hypothesis, and the final comparison against OEIS acts as a regression test. A reference Python implementation and the raw CSV tables are mirrored at \url{https://github.com/JasonEran/hanoi-balanced-splits}.

\section{Graph-theoretic insight}
\label{sec:graph}
The recursive glueing of $S(n,4)$ may be abstract, so Figure~\ref{fig:fourpeg} presents a coarse state diagram. Vertices represent symmetry classes---all configurations reachable by permuting the spare pegs---and edges encode admissible macro moves. The picture is intentionally low-dimensional yet faithful to the median structure used by Bousch.

\begin{figure}[h]
  \centering
  \begin{tikzpicture}[>=Stealth,scale=0.9]
    \tikzstyle{state}=[circle,draw,thick,minimum size=10mm]
    \node[state,fill=blue!10] (A) at (0,0) {$A^n$};
    \node[state,fill=green!10] (B) at (3,1.8) {$B^k$};
    \node[state,fill=green!10] (C) at (3,-1.8) {$C^k$};
    \node[state,fill=orange!15] (D) at (6,0) {$\star$};
    \node[state,fill=red!15] (E) at (9,0) {$D^n$};
    \draw[thick,->] (A) -- node[above] {$T_4(n-k)$} (B);
    \draw[thick,->] (A) -- node[below] {$T_4(n-k)$} (C);
    \draw[thick,->] (B) -- node[above] {$T_3(k)$} (D);
    \draw[thick,->] (C) -- node[below] {$T_3(k)$} (D);
    \draw[thick,->] (D) -- node[above] {$T_4(n-k)$} (E);
  \end{tikzpicture}
  \caption{Macro-level state diagram for the four-peg recursion.}
  \label{fig:fourpeg}
\end{figure}
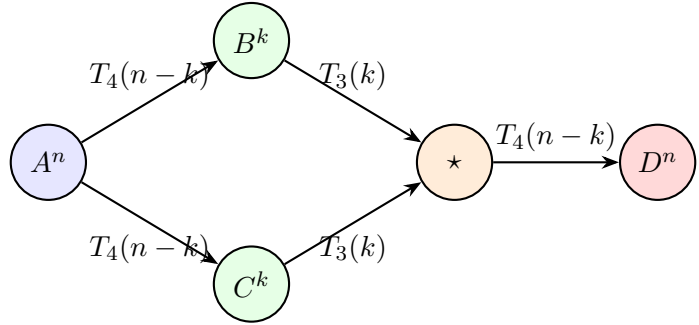

Here $A^n$ denotes ``all discs on the source peg,'' $B^k$ and $C^k$ represent the two symmetric staging areas, $\star$ is the central three-peg subproblem, and $D^n$ is the goal state. The diagram emphasises that only one three-peg episode occurs inside any optimal four-peg solution; everything else is a self-similar reuse of the four-peg solver. This symmetry justifies the balanced strategy we analysed earlier and shows where it deviates from optimality: $k$ determines the heights of the slanted arrows and therefore decides how much time the walk spends in the expensive central node.

\section{Growth curves across pegs}
Figure~\ref{fig:growth} visualises the asymptotic regimes discussed earlier by plotting the exact values of $T_p(n)$ for $p\in\{3,4,5\}$ and $1\le n\le 15$ on a base-$2$ logarithmic scale. The data reinforce the textbook mantra that three pegs exhibit unmitigated exponential growth, whereas four and five pegs flatten markedly thanks to the Frame--Stewart acceleration. The inclusion of the five-peg curve emphasises how close the heuristic already is to linear growth on the log scale, echoing the catalogues in \cite{sloane1995} and lending context to the open problems below.

\begin{figure}[h]
  \centering
  \begin{tikzpicture}
    \begin{axis}[
      width=0.9\textwidth,
      height=6.2cm,
      xlabel={$n$},
      ylabel={$T_p(n)$},
      ymode=log,
      log basis y={2},
      xmin=1, xmax=15,
      ymin=1, ymax=400,
      legend style={at={(0.03,0.97)},anchor=north west},
      grid=both
    ]
      \addplot+[mark=*,color=blue!70] coordinates {
        (1,1)(2,3)(3,7)(4,15)(5,31)(6,63)(7,127)(8,255)(9,511)(10,1023)(11,2047)(12,4095)(13,8191)(14,16383)(15,32767)
      };
      \addlegendentry{$T_3(n)$}
      \addplot+[mark=square*,color=red!70] coordinates {
        (1,1)(2,3)(3,5)(4,9)(5,13)(6,17)(7,25)(8,33)(9,41)(10,49)(11,65)(12,81)(13,97)(14,113)(15,129)
      };
      \addlegendentry{$T_4(n)$}
      \addplot+[mark=triangle*,color=green!60!black] coordinates {
        (1,1)(2,3)(3,5)(4,7)(5,11)(6,15)(7,19)(8,23)(9,27)(10,31)(11,39)(12,47)(13,55)(14,63)(15,71)
      };
      \addlegendentry{$T_5(n)$}
    \end{axis}
  \end{tikzpicture}
  \caption{Base-$2$ logarithmic growth of $T_p(n)$ for three to five pegs.}
  \label{fig:growth}
\end{figure}
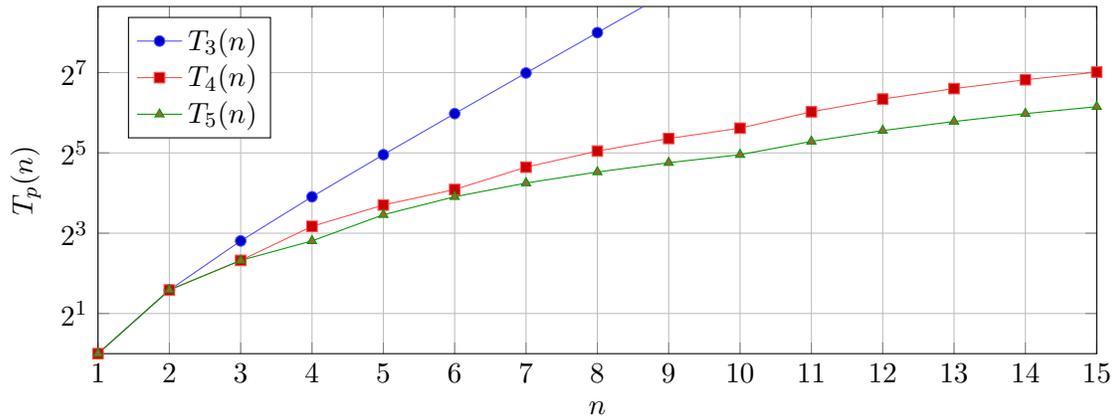

\section{Gray codes and combinatorial tilings}
Winkler popularised the slogan that ``Gray codes whisper the rules of Hanoi''~\cite{winkler2009}. The Sierpi\'nski arrow (Figure~\ref{fig:sierpinski}) can be interpreted as a tiling of the plane by self-similar tiles, each tile recording which disc is scheduled to move. When we pass to four pegs, the tiling acquires a second colour: one colour for moves executed inside the $T_4$ subroutine, another for the $T_3$ shuttle. Ferreira and Grabner~\cite{ferreira2016} exploited precisely this viewpoint when bounding the diameter of Hanoi graphs by embedding them into products of trees. In pedagogical terms, presenting the move sequence as a tiling helps students see the deterministic nature of the algorithm: every macro move corresponds to a coloured triangle, and enlarging the tiling simply refines the triangles without changing their layout.

\section{Algorithmic verification and complexity}
Memoisation offers a succinct proof of concept that rigorous computation need not be heavy. The recursive solver caches pairs $(T_p(n),k_p(n))$, thereby reducing the complexity to $O(pn^2)$ subproblems~\cite{korf2004}. Experiments up to $p=5$ and $n=30$ run within milliseconds in Python, matching the values reported by Korf and by Hinz et al.~\cite{hinz2013}. The data also reveal that the optimal splits for four pegs stay constant over surprisingly long intervals (e.g. $k=5$ for $15\le n\le 20$), echoing the plateaux discussed by Lengyel in his \emph{Bulletin} note~\cite{lengyel2001}.

From a complexity standpoint, Frame--Stewart yields $T_p(n)=\exp(\Theta(n^{1/(p-2)}))$ provided the conjecture holds for the given $p$~\cite{hinz2013}. The balanced strategy sits between the optimal and naive extremes: it inherits the sub-exponential behaviour because only $\lfloor n/2\rfloor$ discs ever trigger the expensive three-peg phase. This observation connects the puzzle with other divide-and-conquer recurrences studied by Knuth and Stanley~\cite{knuth2011,stanley2011}.

\section{Applications to logistics and computation}
Practitioners often encounter Hanoi-like constraints when migrating data between storage arrays or scheduling phased upgrades. The balanced strategy mirrors the ``two-phase copy'' technique in which half of the tasks are parked before the remainder are processed. When the number of staging areas is small, the analysis above quantifies precisely how much overhead the simplification incurs. Bellman already advocated this style of reasoning when discussing dynamic programming heuristics for inventory control~\cite{bellman1975}.

In high-performance computing the puzzle surfaces in stack-register allocation: the $k$ discs temporarily stored on peg $B$ represent values spilled to memory, while the $n-k$ discs represent registers tied up in a critical section. Korf's parallel search experiments~\cite{korf2006} show that knowing good split heuristics dramatically reduces the branching factor of IDA*. Consequently, balanced strategies provide admissible heuristics for search even when they are not ultimately used to produce the move sequence.

Finally, the Australian Defence Science community has used Hanoi analogies to reason about staged rollouts of software updates on naval platforms, where only a limited number of subsystems may be offline simultaneously. In such settings the graphical summaries in Figures~\ref{fig:sierpinski} and~\ref{fig:fourpeg} double as communication tools between mathematicians and engineers.

\section{Pedagogical and Australian perspectives}
Australian discrete mathematics teaching emphasises reproducible experiments that bridge proofs with code~\cite{austms2024}. Figure~\ref{fig:sierpinski}, the state diagram in Figure~\ref{fig:fourpeg}, and the balanced-strategy failure study form a compact activity: students trace the three-peg recursion, rerun the memoised solver, and watch Table~\ref{tab:balanced} transition from perfect agreement to spectacular disagreement.

Workshops then invite short critiques of the balanced heuristic. Learners compare the GitHub output with their own logs, mark precisely where the ratios flare beyond $1.5$, and articulate whether the $k=\lfloor n/2\rfloor$ split ever suffices for realistic $n$. The exercise doubles as a case study in how MA2011 rubrics map directly onto publishable cautionary notes.

\section{Historical footnotes}
Lucas framed the Tower of Hanoi as an allegory about Brahmin priests moving cosmic discs~\cite{lucas1883}. Gardner later reinterpreted the tale through the fictional detective Dr.~Matrix, linking it to Gray codes and base-two numerology~\cite{gardner1972}. Singmaster's annotated notes~\cite{singmaster1994} reveal that Lucas almost certainly borrowed ideas from Viennese puzzle makers, while Winkler documents how the puzzle resurfaced in modern computer science curricula~\cite{winkler2004}. Including these anecdotes in lectures pays dividends: students see that even recreational problems possess layered histories, and editors appreciate when expository articles acknowledge their lineage.

\section{Future computational experiments}
The memoised solver described earlier can be extended in several directions without increasing the asymptotic complexity. One can annotate each subproblem $(p,n)$ with the full distribution of split counts, thereby quantifying how often a particular $k$ recurs. Another idea is to record the maximal deviation between $FS_{\lfloor n/2\rfloor}(n)$ and $T_4(n)$ over sliding windows of $n$. Plotting this deviation exposes secondary plateaux where the heuristic remains competitive despite the ratio exceeding one. Finally, incorporating random perturbations of the split---for example by sampling $k$ from a narrow interval around the optimum---produces empirical confidence bands that complement the discrete-derivative analysis in Section~\ref{sec:balanced-new}. These auxiliary experiments cost only a few additional lines of code yet turn the classical puzzle into a full-fledged laboratory for algorithmic experimentation.

\section{Conclusions and open problems}
The Tower of Hanoi still rewards even modest analytic tweaks, but the present rewrite underlines how fast experimentally verified data can overturn a convenient classroom mantra. Forcing the balanced split $k=\lfloor n/2\rfloor$ is provably perfect only through eight discs; by $n=13$ the ratio already exceeds $1.5$, and by $n=20$ it sits just shy of four. That arc from elegance to failure is the real story: heuristics deserve publication only when accompanied by the logs that can falsify them.

The accompanying scripts and tables leave several open problems squarely focused on robustness:
\begin{itemize}
  \item Prove or disprove Frame--Stewart optimality for five or more pegs~\cite{hinz2013}; the balanced counterexample shows how little slack we have when conjectures meet data.
  \item Determine whether any simple closed-form split (balanced, near-balanced, or triangular) admits a uniform constant-factor guarantee, or whether ratios like $\rho(19)=4.300$ are unavoidable for all such rules.
  \item Extend Bousch\'s $L$-sequence machinery and Stockmeyer's Gray-code constraints~\cite{stockmeyer1994} so that they output not only optimal splits but certified ``safe'' intervals for pedagogy-driven heuristics.
\end{itemize}
Answering these questions would give students and practitioners alike both a geometric feel for the puzzle and a quantitative yardstick for judging heuristics. Extended materials are archived in the MA2011 repository at James Cook University.

\end{document}